\documentclass[11pt,english]{smfart}

\usepackage{amsmath,xy,leftidx,amsthm,sfheaders}
\usepackage{xspace}
\usepackage[psamsfonts]{amssymb}
\usepackage[latin1]{inputenc}
\usepackage{graphicx,color}
\usepackage{hyperref}

\theoremstyle{remark}

\newtheorem{remark}{\bf Remark}

\theoremstyle{plain}
\newtheorem{theorem}{\bf Theorem}
\newtheorem*{Theorem1}{\bf Main Theorem}
\newtheorem{proposition}[theorem]{\bf Proposition}

\newtheorem*{corollary*}{\bf Corollary}
\newtheorem{lemma}[theorem]{\bf Lemma}
\newtheorem{corollary}[theorem]{\bf Corollary}

\newtheorem*{conjecture*}{\bf  Conjecture}

\def\A{{\mathbb A}}
\def\C{{\mathbb C}}
\def\R{{\mathbb R}}
\def\Z{{\mathbb Z}}

\def\Q{{\mathbb Q}}

\def\D{\mathbb{D}}
\def\KK{\mathbb{K}}
\def\LL{\mathbb{L}}
\def\p{\mathbb{P}}
\def\N{{\mathbb N}}

\def\sf{{\mathsf f}}
\def\sa{{\mathsf a}}
\def\sb{{\mathsf b}}
\def\sc{{\mathsf c}}
\def\sP{{\mathsf P}}
\def\sQ{{\mathsf Q}}

\def\O{\mathcal{O}}
\def\cO{\mathcal{O}}

\def\cT{{\mathcal T}}

\def\a{\alpha}

\DeclareMathOperator{\an}{an}

\DeclareMathOperator{\hdiv}{\widehat{div}}

\DeclareMathOperator{\diam}{diam}

\def\and{{\quad\text{and}\quad}}

\begin{document}

\title[Continuity of the Green function]{Continuity of the Green function in meromorphic families of polynomials}
\author{Charles Favre}
\address{CMLS, \'Ecole polytechnique, CNRS, Universit\'e Paris-Saclay, 91128 Palaiseau Cedex, France}
\email{charles.favre@polytechnique.edu}
\author{Thomas Gauthier}
\address{LAMFA, Universit\'e de Picardie Jules Verne, 33 rue Saint Leu, 80039 Amiens Cedex}
\address{CMLS, \'Ecole polytechnique, CNRS, Universit\'e Paris-Saclay, 91128 Palaiseau Cedex, France}
\email{thomas.gauthier@u-picardie.fr}

\thanks{First author is supported by the ERC-starting grant project "Nonarcomp" no.307856, both authors are partially supported by ANR project ``Lambda'' ANR-13-BS01-0002}

\date{\today}

\begin{abstract}
We prove that along any marked point  the Green function of a meromorphic family of polynomials parameterized by the punctured unit disk explodes exponentially fast near the origin with a continuous error term.
\end{abstract}

\maketitle
\tableofcontents

\section*{Introduction}
 
Our aim is to analyze in detail the degeneration of the Green function of a meromorphic family of polynomials. Our main result is somehow technical but is the key for applications 
in the study of algebraic curves in the parameter space of polynomials using technics from arithmetic geometry. In particular it applies to the dynamical Andr\'e-Oort conjecture
for curves in the moduli space of polynomials~\cite{BD,ghioca-ye,favre-gauthier} and to the problem of unlikely intersection~\cite{BD-unlikely}. We postpone to another paper these applications.

\medskip

Let us describe our setup. We fix any algebraically closed complete metrized field $(k,|\cdot|)$. In the applications we have in mind, the field $k$ is either the field of complex numbers, or
the $p$-adic field $\C_p$. In particular, the norm $|\cdot|$ may be either Archimedean or non-Archimedean. 

Let $P$ be any degree $d\ge 2$ polynomial with coefficients in $k$. Recall that the sequence of functions
$\frac1{d^n} \log\max \{ 1, |P^n|\}$ converges uniformly on $k$ to a continuous function $g_P$ which satisfies the invariance property $g_P \circ P = d g_P$.  
The function $g_P$ is also continuous as a function of $P$ when the polynomial ranges over the set of polynomials of degree $d$.
Our analysis gives precise informations on the behaviour of $g_P$ when $P$ degenerates.

\medskip

More precisely, denote by $\D = \{ |z| <1\}$ the open unit disk in the affine line over $k$, 
and let $\cO(\D)$ be the set of analytic functions on $\D$. A function $f$ belongs to $\cO(\D)$ if it can be expanded as a power series $f(t) = \sum_{i\ge 0} a_i t^i$
with the condition $\sum_{i\ge 0} |a_i|\, \rho^i< \infty$ (resp. $|a_i| \rho^i \to 0$) for all $\rho<1$ when $k$ is Archimedean (resp. non-Archimedean). 
Observe that a function $f$ belongs to $\cO(\D)[t^{-1}]$ iff it is a meromorphic function on $\D$ with a single pole at $0$. 
Now it follows from\cite[\S3]{demarco} that  $g_{P_t} (a(t))\sim \alpha \log|t|^{-1}$ for some non negative constant $\alpha$ (see also~\cite{favre} for a generalization of this fact to higher dimension).

\begin{Theorem1}\label{tm:continuity}
For any meromorphic family $P_t\in\cO(\D)[t^{-1}][z]$ of degree $d\ge2$ polynomials and for any function $a(t)\in\mathcal{O}(\D)[t^{-1}]$, 
there exists a non-negative rational number $\alpha\in \Q_+$ such that the function
$$
h(t) := g_{P_t} (a(t)) - \alpha \log |t|^{-1}
$$
extends continuously across the origin.  Moreover, one of the following occurs: 
\begin{enumerate}
\item
there exists an affine change of coordinates depending analytically on $t$ conjugating $P_t$
to an analytic family $Q_t$ of polynomials (so that $\deg (Q_0) =d$);
\item
the constant $\alpha$ is positive and $h$ is harmonic in a neighborhood of $0$;
\item
the constant $\alpha$ vanishes, and $h(0)=0$.
\end{enumerate}
\end{Theorem1}
This result was previously known only for polynomials of degree $3$ with a marked critical point, see~\cite[Theorem 3.3]{ghioca-ye}. Indeed the core of the proof is the continuity of $h(t)$
when the constant $\alpha$ is zero and we follow their line of arguments at this crucial step.

Observe that our main theorem fails for families of rational maps. 
DeMarco and Okuyama have recently constructed a family of rational maps of degree $2$ with a critical marked point for which the continuity statement does not hold. The rationality of the coefficient $\alpha$ also fails for rational maps, as shown by DeMarco and Ghioca in~\cite{DeMarco-Ghioca}.

\medskip

Suppose that $k$ is of characteristic zero.
Recall that the equilibrium measure $\mu_P$ of  a polynomial $P$ of degree $d\ge 2$ is the limit of the sequence of pull-backs
$\mu_P := \lim_{n\to\infty} d^{-n} P^{n*} \delta_x$ for all  $x$ but at most two exceptions, see~\cite{MR2578470}. It is a $P$-invariant probability measure whose support is the Julia set, 
and it integrates the logarithm of the modulus of any polynomial. 
In particular,  one can define the Lyapunov exponent $L(P)$ as the integral $L(P) = \int \log|P'| \, d\mu_P$. 
By the Manning-Przytycky's formula (obtained by~\cite[\S 2]{okuyama} in the non-Archimedean case), the Lyapunov exponent satisfies
the formula
$$
L(P_t) = \log |d| + \sum_{i=1}^{d-1}  g_{P_t}(c_i)~,
$$
where $c_1, \cdots , c_{d-1}$ denote the critical points of $P_t$ counted with multiplicity. 

Our Main Theorem then implies the
\begin{corollary}\label{cor1}
For any meromorphic family $P_t\in\cO(\D)[t^{-1}][z]$ of degree $d\ge2$ polynomials defined over a field of characteristic zero, 
there exists a non-negative rational number $\lambda$ such that 
the function $t\mapsto L(P_t) - \lambda \log |t|^{-1}$
extends continuously through the origin. 

Moreover we have $\lambda >0$ except if there exists an affine change of coordinates depending on $t$ conjugating $P_t$
to an analytic family of polynomials.
\end{corollary}

\smallskip

Let $C$ be a smooth connected affine curve defined over a  number field $\KK$. An \emph{algebraic family} $P$ parametrized by $C$ is a map $P:C\times \A^1_\KK\to C\times \A^1_\KK$ of the form  $P(t,z)=(t,P_t(z))$ where $P_t$ is a polynomial of degree $d\ge 2$ for all $t$.

A pair $(P,a)$ with $a\in\KK[C]$ is said to be \emph{isotrivial} if there exists  a finite field extension $\LL/\KK$, a finite branched cover $p:C'\to C$ defined over $\LL$ and a map $\phi:C'\times\A^1\to C'\times \A^1$ of the form $\phi(t,z)=(t,\phi_t(z))$ where $\phi_t$ is an affine map for all $t$ such that both $\phi_t\circ P_t\circ \phi_t^{-1}$ and $\phi_t(a(t))$ are independent of $t$.
Finally $a$ is \emph{persistently preperiodic} on $C$ if there exist two integers $n>m\geq0$ such that $P^n(a)=P^m(a)$ (as regular functions on $C$).

Recall from~\cite{Silverman} that for any $t$ in the algebraic closure $\bar{\KK}$ of $\KK$, one can build a canonical height function $\hat{h}_{P_t}: \bar{\KK} \to \R_+$ such that  $\hat{h}_{P_t}\circ P_t = d  \, \hat{h}_{P_t}$ and 
$\hat{h}_{P_t}(b) =0$ iff $b$ is preperiodic.

Let us define the height function $h_{P,a}$ on $C(\bar{\KK})$ by setting
\[h_{P,a}(t):=\hat{h}_{P_t}(a(t)), \text{ for all } t\in C(\bar{\KK}).\]
Note that  $h_{P,a}(t)=0$ if and only if $a(t)$ is preperiodic under iteration of $P_t$.

Our next result shows that this height function is in fact determined by nice geometric data in the sense of Arakelov theory.
We refer to e.g. the survey~\cite{ACL2}  for basics on metrizations on line bundles and their associated height function.

Denote by $\bar{C}$ the (unique up to marked isomorphism) smooth projective curve containing $C$ as an open Zariski dense subset.
\smallskip

\begin{corollary}\label{cor2}
Let $C$ be an irreducible affine curve defined over a number field $\KK$. Let $P$ be an algebraic family parametrized by $C$ and pick any marked point $a\in\KK[C]$. Assume that the pair $(P,a)$ is not isotrivial and that $a$ is not persistently preperiodic on $C$. 

Then there exists an integer $q\geq1$ such that the height function $q\cdot h_{P,a}$ is induced by an adelic semipositive continuous metrization on some ample line bundle $\mathcal{L}\to \bar{C}$.

Moreover, the global height of the curve $\bar{C}$ is zero, i.e. $h_{P,a}(\bar{C})=0$.
\end{corollary}



\begin{center}
$\diamond$
\end{center}

Let us explain how we prove our Main Theorem. 
Basic estimates using the Nullstellensatz imply the existence of a constant $C>0$ such that 
\begin{equation}\label{eq:basic}
\left|
\frac1d \, \log \max \{ 1, |P_t(z)|\} - \log \max \{ 1, |z|\} 
\right|
\le C \, \log|t|^{-1}
\end{equation}
for all $|t|<1/2$ and for all $z\in k$.  Using~\eqref{eq:basic}, it is not difficult to see that  $g_{P_t} (a(t))= \alpha \log |t|^{-1} + o(\log|t|^{-1})$ for some $\alpha\in \R_+$
(see~\cite[\S3]{demarco}). To get further, we shall interpret the constant $\alpha$ in terms of the dynamics of  the polynomial with coefficients in $k((t))$  naturally induced by the family $P_t$.

We endow the ring $\cO(\D)[t^{-1}]$ with the $t$-adic norm $|\cdot|_r$ whose restriction to $k$ is trivial and normalized by $|t| = e^{-1}$.
Observe that the completion of its field of fraction is the field of Laurent series $(k((t)), |\cdot|)$. We may then view the family $P_t$ as a polynomial $\sP$ with coefficients in 
the complete metrized $(k((t)), |\cdot|)$ and consider its dynamical Green function $g_\sP : k((t)) \to \R_+$. The marked point $a$ gives rise to a point $\sa \in k((t))$, and
it follows from the analysis developed in~\cite{favre} that $\alpha = g_{\sP}(\sa)$. 

\smallskip

Let us first consider  the case $\alpha >0$. The point $\sa$ then belongs to the basin of attraction at infinity of $\sP$ which implies $a(t)$ to belong also  to the basin of attraction
at infinity of $P_t$ for all $t$ small enough. To conclude one then uses the fact that the Green function is the logarithm of the B\"ottcher coordinate and expand this coordinate as an analytic function 
in the two parameters $z$ and $t$.  This strategy was made precise in degree $3$ in~\cite{favre-gauthier} and~\cite{ghioca-ye}, and we write the details here in arbitrary degree for the convenience of the reader. 

When $\alpha=0$ the point $\sa$ lies in the filled-in Julia set of $\sP$. Since $k((t))$ is discretely valued,  results of Trucco~\cite{trucco}
give strong restrictions on the orbit of $\sa$. In \S\ref{sec:compact}, we give a direct argument showing that either $\sa$ lies in the Fatou set of $\sP$ and belongs to a preperiodic ball; or $\sa$ lies in the Julia set of $\sP$ and the closure of 
its orbit is compact in $k((t))$. 

In the former case, one can make a change of coordinates (depending on $t$) and assume that $ P_t(z)  =Q(z) + t\, R_t(z)$ where $\delta = \deg (Q) \le d$. 
When $\delta =d$ the family of polynomial $P_t$ degenerates to a degree $d$ polynomial and the Green function $g_{P_t}(z)$ is continuous both in $z$ and $t$. 
Otherwise $\delta <d$, and direct estimates show that $  g_{P_t} (a(t)) = o(1)$ as required. 

In the latter case, the estimates are more delicate and we follow the arguments of Ghioca and Ye in~\cite[Theorem 3.1]{ghioca-ye}. The key observation is the following. Since the closure of the orbit of $\sa$ is compact, for any integer $l$ there exists a finite collection of polynomials $Q_1, \cdots, Q_N$ such that for all $n$, one has
$P^n_t(a(t)) = Q_{i_n}(t) + o(t^{l+1})$ for some $i_n \in \{ 1, \cdots , N\}$.  The proof of $  g_{P_t} (a(t)) = o(1)$ uses in a subtle way this approximation result together with~\eqref{eq:basic}.

\medskip

\noindent{\bf Ackowledgement}: we extend our warmful thanks to Laura DeMarco and Yusuke Okuyama for sharing with us
their results on the (dis)-continuity of the error term for general rational families.

%
%
%


\section{Compact orbits of polynomials}\label{sec:compact}

In this section we fix a discrete valued complete field $(L, |\cdot|)$. In our applications we shall take $L = k((t))$
endowed with the $t$-adic norm  normalized by  $|t| = e^{-1}<1$ where $k$ is an arbitrary field. 

\subsection{A criterium for the compactness of polynomial orbits}
Let $P$ be any polynomial of degree $d\ge 2$ with coefficients in $L$. Recall that one can find a positive constant $C>0$ such that
\begin{equation}\label{eq:base}
\frac1{C} \le
\frac{\max \{ 1, |P(z)|\}}{\max\{ 1, |z|\}^d}
\le C
\end{equation}
for all $z \in L$. It follows that the sequence $ \frac1{d^n} \log \max \{ 1, |P^n|\}$ converges uniformly on $L$ to a  continuous function $g_P: L \to \R_+$
such that $g_P \circ P = d \, g_P$ and $g_P(z) = \log |z| + O(1)$.

\begin{theorem}\label{thm:trucco}
Suppose $P \in L [z]$ is a polynomial of degree $d\ge2$, and $a$ is a point in $L$. 
Then $g_P(a) \in \Q_+$ and one of the following holds.
\begin{enumerate}
\item
The iterates of $a$ tend to infinity in $L$, and $g_P(a) >0$.
\item
The point $a$ lies in a pre-periodic ball $B$ whose radius lies in $|L^*|$, and $g_P(a) =0$.
\item
The closure of the orbit of $a$ in $L$ is compact in $(L, |\cdot|)$, and $g_P(a) =0$.
\end{enumerate}
\end{theorem}

\begin{remark}
This fact is a direct consequence of the results of Trucco~\cite[Proposition 6.7]{trucco} when the characteristic of $k$ is zero. We present here a simple proof which does not rely on the delicate combinatorial analysis
done by Trucco in his paper and does not use any assumption on $k$.
\end{remark}

\subsection{The tree of closed balls}

Let $\cT$ be the space of closed balls in $L$:  a point in $\cT$ is a set of the form $\overline{B(z_0,r)}:=\{ z \in L, |z-z_0| \le r\}$ for some $z_0$ and some $r\in |L^*|$. 
The map sending a point $z\in L$ to the ball of center $z$ and radius $0$ identifies $L$ with a subset of $\cT$.
We endow $\cT$ with the weakest  topology making all evaluation maps $Q \mapsto |Q(x)| := \sup_x |Q|$ continuous for all $Q\in L[T]$.
By Tychonov, for any $r\ge 0$ the set  $\{ x\in \cT , |T(x)|\le r\}$ is  compact for this (weak) topology.

\smallskip

Observe that for any closed ball $x= \overline{B(z,r)} \in \cT$ with $r\in |L^*|$, we have $\diam(x)  := \sup_{z,z'\in x} |z -z'| =r$.
Since $(L,|\cdot|)$ is non-Archimedean,  any point $z'\in x$ is a center for $x$ so that  $ x= \overline{B(z',\diam(x))}$. Also any two balls $x, x'\in \cT$
are either contained one into the other or disjoint. When $x$ is contained in $x'$, we may write
$x= \overline{B(z,r)}$ and $x'=\overline{B(z,r')}$ for some $r, r'\in |L^*|$, and one sets $d(x',x) = |r' - r| =  |\diam (x') - \diam(x)|$. 

\smallskip

Denote by $\le$ the order relation on $\cT$ induced by the inclusion, i.e. $x \le x'$ iff the ball $x$ is included in $x'$. 
For any two balls $x = \overline{B(z,r)}$, $x'=\overline{B(z',r')}\in \cT$, we let $x \vee x'$ be the smallest closed ball containing both $x$ and $x'$, so that 
$$x \vee x' = B(z, \max\{ r, r', |z-z'|\}) = B(z', \max\{ r, r', |z-z'|\})~.$$ 
The distance between any two closed balls $x, x'$ is now defined as 
$$d( x, x') = \max\{ |\diam (x \vee x') - \diam(x)| , |\diam (x \vee x')- \diam(x')|\} \in |L^*|~.$$
The restriction of the distance $d$ to $L$ is the ultra-distance induced by the norm $|\cdot|$. 

In the sequel  the strong topology refers to the topology on $\cT$ induced by the ultra-distance $d$. 
It is not locally compact except if the residue field $\tilde{L} := \{ |z| \le 1\} / \{ |z| <1\}$ is finite. 

\smallskip

Let $x_n$ be a sequence of points of norm $\le 1$. Its residue classes $\tilde{x}_n$ are by definition their images in $\tilde{L}$ under the natural projection
$\{ |z| \le 1\} \to \tilde{L}$. If the points $\tilde{x}_n$ are all distinct then for any polynomial $Q = \sum_{k\le d} a_k T^k \in L[T]$ we have 
$|Q(x_n)| = \max \{  |a_k|\}$ for all $n$ large enough so that $x_n$ converges in the weak topology to the point $\overline{B(0,1)}$.

 \begin{proposition}  \label{prop0}
Let $F$ be any bounded infinite subset of $L$ so that $\sup_F |T|< \infty$. Then
\begin{enumerate}
\item
either the weak closure of $F$ is strongly compact;
\item
or one can find a  closed ball of positive radius $x\in \cT$ containing infinitely many points $x_n \in F$ such that $x_n \to x$.
\end{enumerate}
\end{proposition}
\begin{proof}
Let $\bar{F}$ be the weak closure of $F$ inside $\cT$. Since $F$ is supposed to be bounded, $\bar{F}$ is weakly compact. 
Observe that for any $z\in L$ and for any $\epsilon>0$, the set $B_\cT(z,\epsilon) := \{ x \in \cT,\, d(x,z) < \epsilon\}$ is 
weakly open since it coincides with $B_\cT(z,\epsilon) = \{ x \in \cT, |(T-z) (x)|< \epsilon\}$.

Assume first that $\bar{F}$ is included in $L$, and take $\epsilon>0$. The previous observation implies that  by weak compactness, the covering of $\bar{F}$ by the family
of balls for all $z\in \bar{F}$ admits a finite subcover. 
It follows that $\bar{F}$ is pre-compact. Since $L$ is complete, $\bar{F}$ is also complete hence strongly compact.

Assume now that $\bar{F}$ contains a point $x$ of positive diameter. Up to making an affine change of coordinates, we may suppose that $x = \overline{B(0,1)}$.
Since $(L,|\cdot|)$ is discrete, we may find $\epsilon>0$ such that $1-\epsilon <|z| <1+\epsilon$ implies $|z| =1$. 
It follows that the set $U = \{ y \in \cT, \, |T(y)|< 1+ \epsilon\}$ is an open neighborhood of $x$, hence it contains infinitely many points of $F$.

Let us now prove that one can find a sequence of points $x_n\in F\cap U$ such that $x_n \to x$. 
We construct the sequence $x_n$ by induction. Choose any $x_1\in x$, and consider the open set
$ U_1:=  \{ y \in \cT, \, |T(y)|< 1 + \epsilon \} \cap  \{ y \in \cT, \, |(T-x_1)(y)|> 1 - \epsilon \}$.
It is an open neighborhood of $x$ which does not contain $x_1$. We may thus find a point $x_2 \in U_1 \cap F$.
Proceeding inductively, we find a sequence of points $x_n$ and open neighborhoods $U_n$ of $x$ such that 
$$
x_n \in U_n :=  \{ y \in \cT, \, |T(y)|< 1 + \epsilon \} \bigcap_{i=1}^{n-1}  \{ y \in \cT, \, |(T-x_i)(y)|> 1 - \epsilon \}~.
$$
The choice of $\epsilon$ implies that the residue classes of $x_n$ and $x_m$ are all distinct which implies $x_n \to x$ as required.
%
\end{proof}

\smallskip

 Let $P(T) = a_0 T^d + \cdots + a_d$ be any polynomial with coefficients in $L$. We define the image of $x = \overline{B(z,r)} \in \cT$ by the formula
 $$
 P(x) = \overline{B(P(z),\max_{i\ge 1}\{ |a_i| r^i\} )}
 $$
 This map is weakly continuous since one has $|Q ( P(x))| = |(Q \circ P) (x)|$ for all polynomial $Q$. It coincides with $P$ on $L$.
 Observe that $P: \cT \to \cT$ preserves the order relation.
 
\smallskip
\begin{remark}\label{rem:tree to berk}
There is a canonical continuous and injective map from $\cT$ into the (Berkovich) analytification of 
$\A^1_L$ sending a closed ball $x\in \cT$ to the multiplicative semi-norm on $L[T]$ defined by  $P \mapsto |P(x)|$. 
This map identifies $\cT$ with the smallest  closed  subset of $\A^1_L$ whose intersection with the set of rigid points in $\A^1_L$ is equal to $L$. 
In the terminology of Berkovich~\cite[\S1]{Berko}, it consists only of type $1$ and type $2$ points.
\end{remark}

\subsection{Proof of Theorem~\ref{thm:trucco}}
When $g_P(a)$ is positive, then $|P^n(a)| \ge (\frac12\, g_P(a))^{d^n}$ for $n$ large enough so that $P^n(a)$ tends to infinity.
In that case one can refine~\eqref{eq:base} since by the non-Archimedean inequality one has 
$|P(z)| = (r |z|)^d$ for all $|z|$ large enough where $r^d$ is the norm of the leading coefficient of $P$ hence belongs to $e^\Z$.
If follows that $g_P(z) = \log (r |z|)$ for $|z|$ large enough, hence
$$
g_P(a) = \frac1{d^N} g_P (P^N(a)) = \frac1{d^N} \log (r |a|) \in \Q_+~.
$$
When $g_P(a) =0$, all iterates of $a$ belong to $\{ g_P = 0 \}$ which is a bounded set of $L$. 
We consider the weak closure $\Omega$ of the orbit of $a$ in the space of balls $\cT$. 
If $\Omega$ consists only of points in $L$, then it is compact in $(L, |\cdot|)$ by Proposition~\ref{prop0} (1), and we are in Case (3).

\smallskip
Otherwise by Proposition~\ref{prop0} (2) there exists a closed ball  $x$ with radius in $|L^*|$, and 
 a strictly increasing sequence $n_i \to \infty$ such that $P^{n_i}(a) \to x$, and $P^{n_i}(a)\in x$ for all $i$. 
Observe that there exists a closed ball containing the orbit of $a$, hence the orbit of $x$ is also bounded. 
Replacing $x$ by a suitable iterate we may assume that $\diam (x) = \max \{ \diam(P^n(x))\}_{n\in \N}$. 

Since $x$ contains both $P^{n_0}(a)$ and $P^{n_1}(a)$, the ball $P^{n_1 -n_0}(x)$ intersects $x$ and thus
is included in $x$. When $P^{n_1 -n_0}(x) =x$, we are in Case (2). When  $P^{n_1 -n_0}(x)$ is strictly included in $x$, then 
 $P^{n_1 -n_0}$ admits an attracting fixed point lying  in $L \cap P^{n_1 -n_0}(x)$ whose basin of attraction contains $x$. 
 This is however not possible because $x$ lies in the closure of the orbit of $a$.


\section{The point $\sa$ has an unbounded orbit}\label{sec:unbounded}

Recall the setting of the statement of the Main Theorem. Let  $P_t(z) = \a_0(t) z^d + \a_1(t) z^{d-1} + \cdots + \a_d(t)$ 
be a meromorphic family of polynomials of degree $d$ with  $\a_i \in \cO(\D)[t^{-1}]$ and $\a_0(t) \neq 0$ for all $t\in \D^*$.
Take also any meromorphic map $a \in  \cO(\D)[t^{-1}]$.

\smallskip

Recall from e.g.~\cite[Lemma 3.3]{demarco} or~\cite[Proposition 4.4]{favre} that the Nullstellensatz implies the existence of a constant $\beta >0$ such that 
\begin{equation}\label{eq:basic-estimate}
 |t|^\beta
\le
\frac{\max \{ 1, |P_t(z)|\}}{\max\{1, |z|^d\}}
\le
 |t|^{ -\beta}
\end{equation}
for all $|t| \le \frac12$ and for all $z\in k$.
Observe that in particular we get
\begin{equation}\label{eq:uniform}
\left|
g_{P_t}(z) - \log \max \{ 1 , |z| \}\right|
\le C \log|t|^{-1}
\end{equation}
for some constant $C>0$, for all $|t|\le \frac12$,  and all $z \in k$.

\smallskip

Since the base field $k$ is supposed to be algebraically closed conjugating $P_t$ by a suitable homothety with coefficient in $k$, we may write
 $\a_0(t) = t^N ( 1 + o(1))$ for some $N\in \Z$.  One can then consider the family of monic polynomials $\tilde{P}_t (z) = \phi_t^{-1} \circ P_{t^{d-1}} \circ \phi_t$ with $\phi_t(z) = \a_0(t^{d-1})^{-1/(d-1)}z$. 
 Observe that $g_{P_{t^{d-1}}}(a(t^{d-1})) = g_{\tilde{P}_t}(\phi_t^{-1} \circ a(t^{d-1}))$ so that the proof of the Main Theorem is reduced to  the case of a family of monic polynomials. From now on we shall therefore assume that $\a_0 \equiv 1$.

%
%
%
%
%
%
%

\begin{lemma}\label{lem:1}
There exists a constant $C>0$ and an integer $N \in \N^*$
such that the following holds. 

For any meromorphic function $a(t) = t^{-l} (1 + h)$ with $l\ge N$, and $h$ is a holomorphic map such that $h(0) =0$ and $\sup_\D |h| \le \frac12$, 
then we have
$$
\left|
\frac1{d} \log |P_t ( a(t))| - \log |a(t)|
\right| 
\le C~,
$$
for any $0<|t| \le \frac12$.
\end{lemma}
\begin{proof}
Pick $N$ large enough such that $\sup_{i, |t| \le \frac12} |t|^N |\a_i (t)| \le 1$. 
For any $l\ge N$, we may then write
$$
P_t ( a(t))  = t^{-ld} \left( (1+h)^d + \a_1 t^l (1+h)^{d-1} + \cdots + \a_0 t^{ld}\right)
$$
so that 
$$
\frac1{d} \log |P_t ( a(t))| - \log |a(t)| = \log \left|
\frac{(1+h)^d + \a_1 t^k (1+h)^{d-1} + \cdots + \a_0 t^{kd}}{1+h}
\right|
$$
is bounded by $ \log ((d+1)2^{d+1})$.
\end{proof}

\begin{proof}[Proof of the Main Theorem in the case $|\sP^n(\sa)| \to \infty$]
Recall that $\sP$ denotes the monic polynomial of degree $d$ with coefficients in $k((t))$ induced by the family $P_t(z) =  z^d + \a_1(t) z^{d-1} + \cdots + \a_d(t)$, 
and $\sa$ is the point in $k((t))$ defined by the meromorphic function $a$.
We endow $k((t))$ with the $t$-adic norm $|\cdot|$ normalized by $|t| = e^{-1}$.

\smallskip

If $\sa$ has an unbounded orbit under $\sP$, then replacing $a$ by $P^{n_0} \circ a$ for $n_0$ sufficiently large
we may suppose that $a$ has a pole of order $l$ which can be taken as large as we wish. 

If $l$ is strictly larger than the constant $C$ appearing in~\eqref{eq:uniform}, then we get
$$g_{P_t}(a(t)) \ge \log \max \{ 1 , |a(t) | \} - C \log|t|^{-1} \ge (l-C) \log |t|^{-1} + O(1)$$
for any $t$ small enough. In particular $g_{P_t}(a(t))$ is positive, and 
$g_{P_t}(a(t)) = \lim_n \frac1{d^n} \log |P^n \circ a(t)|$.

Now suppose $l \ge N$ so that we may apply the previous lemma. 
It follows that the sequence of functions $\frac1{d^n} \log |P^n \circ a| - \log |a|$ converges uniformly in $\D^*_{1/2}$ to 
a function $\varphi$ which is necessarily harmonic and bounded. This function thus extends to the origin and remains harmonic. 
We conclude that $g_{P_t}(a(t))  = \varphi(t) + \log|a|$ which shows that we are in Case 2 of the Main Theorem.
\end{proof}

\section{The orbit of $\sa$ is bounded and lies in the Fatou set}\label{sec:fatou}
We suppose now that $|\sP^m(\sa)|$ remains bounded and $\sa$ belongs to a fixed closed ball $B := \overline{B(\sb, \rho)}$ with $\sb \in k((t))$ and $\rho  = r^n$ for some $n\in \Z$. 
Observe that we may take $\sb$ to be a Laurent polynomial $b(t)$, and conjugating $P_t $ by $\phi_t(z) = t^n ( z + b(t))$, we may thus suppose that $B$ is the closed unit ball.

In that case, we can write $P_t(z)=Q(z)+tR_t(z)$ for some polynomial $Q\in k[z]$ with $1\leq \delta:=\deg(Q)\le d$ and $R_t(z)\in k((t))[z]$.
When $\delta =d$ then we are in the Case 1 of the Main Theorem.

When $\delta <d$, we shall prove that we fall in Case 3. 
To see this, observe first that there exists $C_1\geq1$ such that
\[\max\left\{1,|P_t(z)|\right\}\leq C_1\cdot \max\left\{1,|z|^\delta,|t|\cdot|z|^d\right\},\]
for all $z\in k$ and all $t\in\D$.
\begin{lemma}\label{lem:6}
 There exists a constant $A\geq 1$ independent of $n$ such that
\begin{equation}\label{eq du lem}
\max\left\{1,|P_t^n(a(t))|\right\}\leq C_1^{1+\cdots+\delta^{n-1}}\cdot A^{\delta^n}
\end{equation}
if $|t|=\left(C_1^{1+\cdots+\delta^{n-2}}\cdot A^{\delta^{n-1}}\right)^{\delta-d}$.
\end{lemma}
We fix some integer $N\ge 2$, and suppose  $|t|=R_N:= \left(C_1^{1+\cdots+\delta^{N-2}}\cdot A^{\delta^{N-1}}\right)^{\delta-d}$.
Under these assumptions,~\eqref{eq:uniform} implies
\begin{eqnarray*}
0\leq g_t(a(t)) & \leq & \frac{1}{d^N}\log^+|P_t^N(a(t))|-C\frac{\log|t|}{d^N}\\
& \leq & \left(\frac{\delta}{d}\right)^N\left(1+\frac{d-\delta}{\delta}C\right)\log\left(C_1^{1/(\delta-1)}A\right)~.
\end{eqnarray*}
Since $\sP^m(\sa)$ belongs to the closed unit ball for all $m$ by assumption, the functions $P^m(a(t))$ are analytic at $0$ so that 
$\frac1{d^m} \log|P^m(a(t))|$ is subharmonic on $\D$ (in the sense of Thuillier when $(k,|\cdot|)$ is non-Archimedean, see e.g.~\cite[\S 3]{thuillier}).
It follows that $g_t(a(t))$ is also subharmonic on $\D$, and the maximum principle
implies
\begin{eqnarray*}
0\leq g_t(a(t)) \leq \left(\frac{\delta}{d}\right)^N\cdot B,
\end{eqnarray*}
for all $|t|\leq R_N$ with $B:=\left(1+\frac{d-\delta}{\delta}C\right)\log\left(C_1^{1/(\delta-1)}A\right)$.

Fix $\varepsilon>0$ and pick $N\geq1$ large enough such that $\left(\frac{\delta}{d}\right)^N\cdot B \leq \varepsilon$. 
We have proved that $0 \le g_t(a(t)) \le \epsilon$ when $|t| \le \eta := R_N$ so that $g_t(a(t))$ is a non-negative subharmonic function on $\D$
which tends to $0$ at the origin.

\begin{proof}[Proof of Lemma~\ref{lem:6}]
 We argue by induction on $n\geq2$. Note that since $a$ is continuous, there exists $A\geq2$ such that $|a(t)|\leq A$ for all $|t|\le \frac12$. Assume that $|t|=A^{\delta-d}$. Then we have
\[\max\left\{1,|P_t(a(t))|\right\}\leq C_1\cdot \max\left\{1,A^\delta,|t|\cdot A^d\right\}=C_1\cdot A^\delta.\]
Now let $|t|=\left(C_1\cdot A^{\delta}\right)^{\delta-d}$. We obtain 
\[\max\left\{1,|P_t^2(a(t))|\right\}\leq C_1\cdot \max\left\{1,(C_1\cdot A^\delta)^\delta,|t|\cdot (C_1\cdot A^\delta)^d\right\}=C_1^{1+\delta}\cdot A^{\delta^2}.\]
Suppose~\eqref{eq du lem} holds for some $n\geq2$. When $|t|=\left(C_1^{1+\cdots+\delta^{n-1}}\cdot A^{\delta^n}\right)^{\delta-d}$, we have
\begin{eqnarray*}
\max\left\{1,|P_t^{n+1}(a(t))|\right\} & \leq & C_1 \max\left\{1,|P_t^n(a(t))|^\delta,|t|\cdot |P_t^n(a(t))|^d\right\}\\
& \leq & C_1\max\left\{1,\left(C_1^{1+\cdots+\delta^{n-1}}\cdot A^{\delta^n}\right)^\delta,|t|\left(C_1^{1+\cdots+\delta^{n-1}}\cdot A^{\delta^n}\right)^d\right\}\\
& \leq & C_1^{1+\cdots+\delta^{n}}\cdot A^{\delta^{n+1}},
\end{eqnarray*}
which concludes the proof.
\end{proof}

%
%

\section{The point $\sa$ lies in the Julia set}\label{sec:julia}
In this section, we complete the proof of the Main Theorem. We apply Theorem~\ref{thm:trucco} and discuss the situation case by case. In Case 1, the arguments of \S\ref{sec:unbounded} enable us to conclude directly. 
 In Case 2, we replace the marked point by $P^m(a)$ for a suitable $m$, and $P$ by a suitable iterate so that $\sa$ belongs to a fixed ball by $\sP$. 
 This case was treated in the previous section. It remains to treat Case 3: the orbit $\sa$ is compact in $k((t))$. 
 
 Conjugating the family by linear maps $\phi_t(z) = t^{-n} z$ with $n$ sufficiently large, we may suppose that the orbit of $\sa$ is included in the closed unit ball. 
 It follows that 
 $$g_m(t):= \frac1{d^m} \log^+|u_m(t)| \, \text{ with } u_m(t):=P_t^m(a(t)) $$ 
 is subharmonic on $\D$ and $g_t(a(t))$ also. 

Replacing $P_t$ by its second iterate, we may and shall assume that $d\ge 3$.
\smallskip

 As $(g_n(t))_n$ converges locally uniformly to $g_t(a(t))$ on $\D^*$ and as $g_t(a(t))$ is bounded on $\D$, there exists a constant $M\geq1$ such that, $\sup_{|t|=1/2}g_n(t)\leq M$ for all $n\geq1$. By the maximum principle, this gives
\[ \sup_{|t|\leq1/2}g_n(t)\leq M.\]

Fix any integer $l\geq C\cdot d$. Since $\sa$ is in the filled-in Julia set of $\sP$, its orbit is bounded. Up to a change of coordinates $z \mapsto t^{-M} z$ with $M$ sufficiently large, we may assume the orbit of $\sa$ lies in the unit ball in $k((t))$. In other words $P^n_t(a(t))$ is analytic at $0$ for all $n$. Observe that the set of balls of radius $r^{-l}$ centered at polynomials
covers the unit ball in $k((t))$. Since the orbit of $\sa$ is compact in $k$, we may thus find a finite collection of polynomials $Q_1, \cdots , Q_N$ such that for any $n$ one can find $i_n \in \{ 1, \cdots ,N\}$ such that  $P^n_t(a(t)) - Q_{i_n}(t) = O(t^l)$.

Let 
\[A:=\max_{1\leq j\leq N}\left\{\sup_{|t|<1}|Q_j(t)|+2\right\}.\]
Fix a very large integer $n_0\geq1$ once and for all. We may thus find $r_0>0$ small enough such that 
\[\sup_{|t|<r_0}|u_{n_0}(t)|\leq \log A.\]
Set $r_{j}:=r_0^{2^{j}}$ for any $j\ge0$, so that $0<r_{j+1}<r_j$ and $r_j\to0$ as $j\to\infty$.
\begin{lemma}\label{lem:2}
For all $j\geq0$, one has
\[\sup_{|t|<r_j}g_{n_0+j}(t)\leq \frac{C_1}{d^{n_0}},\]
with $C_1=\frac{d}{d-1}\log(3A)$.
\end{lemma}
Applying~\eqref{eq:basic} and the maximum principle, we find
\[0\leq g_{n_0+j+\ell}(t)\leq \sup_{|\tau|<r_j}g_{n_0+j}(\tau)-\frac{C}{d^{n_0+j}}\log r_j\leq\frac{C_2}{d^{n_0}}\left(1-\left(\frac{2}{d}\right)^j\log r_0\right),\]
for any $|t|<r_j$, where $C_2:=\max(C,C_1)$.

\smallskip

Pick now $\varepsilon>0$. We may choose $n_0$ be large enough such that $\frac{C_2}{d^{n_0}}\leq\varepsilon/2$. We then fix $j\geq0$ large enough (depending on $r_0$, hence on $n_0$) so that $\left(1-\left(\frac{2}{d}\right)^j\log r_0\right)\leq 2$. The previous estimate implies
\[0 \le g_n(t) \leq \varepsilon\]
for all $|t|<r_j$ and all $n\geq n_0+j$. Letting $n$ tend to infinity we finally obtain 
\[0 \le g_t(a(t)) \leq \varepsilon\] for all $|t|<r_j$ which concludes the proof.


\begin{proof}[Proof of Lemma~\ref{lem:2}]
It is sufficient to show by induction on $j\geq0$ that 
\[\sup_{|t|<r_{j+1}}g_{n_0+j+1}(t)\leq \frac{\log(3A)}{d^{n_0+j+1}}+ \sup_{|s|<r_j}g_{n_0+j}(s).\]
By assumption, for all $j\geq1$, there exists $1\leq i_j\leq N$ such that the function
\[\frac{u_{n_0+j}(t)-Q_{i_j}(t)}{t^{l}}\]
is analytic on $\D$. The maximum principle implies for any $0<r<1$ the estimate
\[\left|u_{n_0+j}(t)-Q_{i_j}(t)\right|\leq \left(A+\sup_{|s|<r}\left|u_{n_0+j}(s)\right|\right)\cdot\left(\frac{|t|}{r}\right)^{l}~ \ \text{for all }|t|<r.\]
In particular, we find
\begin{eqnarray*}
\sup_{|t|<r_{j+1}}\, \left|u_{n_0+j+1}(t)\right| & \leq & A+\left(A+\sup_{|s|<r_j}\left|u_{n_0+j+1}(s)\right|\right)\cdot\left(\frac{r_{j+1}}{r_j}\right)^{l}\\
& \leq & 2A+\left(\sup_{|s|<r_j}\left|u_{n_0+j+1}(s)\right|\right)r_j^{l},
\end{eqnarray*}
hence
$$
\sup_{|t|<r_{j+1}}\, \max \{ 1, \left|u_{n_0+j+1}(t)\right|\}
\le (3A)\, \sup_{|s|<r_j}\max \{ 1,  \left|u_{n_0+j+1}(s)\right|\,  r_j^{l}\}~.
$$
When $\sup_{|s|<r_j}  \left|u_{n_0+j+1}(s)\right|\,  r_j^{l} \le 1$, we get
$$
\sup_{|t|<r_{j+1}}\, g_{n_0+j+1}(t) \le \frac{\log(3A)}{d^{n_0+j+1} } \le  \frac{\log(3A)}{d^{n_0+j+1}}+ \sup_{|s|<r_j}g_{n_0+j}(s)~,
$$
as required. Otherwise, we have
\begin{eqnarray*}
\sup_{|t|<r_{j+1}}\, g_{n_0+j+1}(t) 
&\le& 
\frac{\log(3A)}{d^{n_0+j+1}} + \frac{l}{d^{n_0+j+1}} \log r_j+ \sup_{|s|<r_j}g_{n_0+j+1}(s)
\\
&
\mathop{\le}\limits^{\text{by~\eqref{eq:basic}}}& 
\frac{\log(3A)}{d^{n_0+j+1}} + \left(\frac{l}{d^{n_0+j+1}}-\frac{C}{d^{n_0+j}}\right) \log r_j+ \sup_{|s|<r_j}g_{n_0+j}(s)
\\
&\le& 
\frac{\log(3A)}{d^{n_0+j+1}}+ \sup_{|s|<r_j}g_{n_0+j}(s)~,
\end{eqnarray*}
since $r_j<1$ and $l\geq Cd$. The lemma follows.
\end{proof}

\begin{remark}
We claim that 
\begin{equation}\label{eq:good exp}
g_{P_t}(a(t))  =  g_\sP(\sa)\, \log|t|^{-1}+ h(t)
\end{equation}
with $h$ continuous. 

When  $|\sP^n(\sa)|$ is bounded, then $g_\sP(\sa)=0$ and the equation follows from the arguments in \S \ref{sec:fatou} and \ref{sec:julia}.
When $|\sP^n(\sa)| \to \infty$ by the invariance of the Green function under iteration it is sufficient to prove~\eqref{eq:good exp} when $a(t)  =t^{-l} (1+h)$ with  $l\ge N$ as in Lemma~\ref{lem:1}.
In that case we have
$$
\frac1{d^n} \log|P^n(a)| = \log |a| + \varphi_n(t)= \log|\sa|\, \log|t|^{-1} + \log |1+h|+ \varphi_n(t)
$$ 
where $\varphi_n$ is a sequence of harmonic functions converging uniformly on $\D_{1/2}$.
We also have  $\frac1{d^n} \log|P^n(a)| = \frac1{d^n} \log|\sP^n(\sa)| \, \log|t|^{-1} + \psi(t)$ where $\psi$ is harmonic, 
hence $\frac1{d^n} \log|\sP^n(\sa)|= \log|\sa|$ for all $n$ and $g_\sP(\sa) = \log |\sa|$ which implies~\eqref{eq:good exp}.
\end{remark}

\section{Degeneration of the Lyapunov exponent}


In this section we prove Corollary~\ref{cor1}. Assume $(k,|\cdot|)$ is an algebraically closed complete metrized field of characteristic zero. 
Fix a meromorphic family $P_t\in\cO(\D)[t^{-1}][z]$ of degree $d\ge2$ polynomials defined over $k$.
Recall that the Lyapunov exponent of $P_t$ is equal to 
\begin{equation}\label{eq:crit mis}
L(P_t) = \log |d| + \sum_{i=1}^{d-1}  g_{P_t}(c_i)~,
\end{equation}
where $c_1, \cdots , c_{d-1}$ denote the critical points of $P_t$ counted with multiplicity. 

To control these critical points when $t$ varies, we observe that the polynomial $P'_t(z)$ is 
a polynomial of degree $d-1$  with coefficients in $\cO(\D)[t^{-1}]$ and
dominant term $da_d(t) z^{d-1}$. It splits over the field of Puiseux series so that one can find Puiseux series
$c_1 (t), \cdots, c_{d-1}(t)$ such that $P'_t(z) = d a_d(t)\, \prod_{i=1}^{d-1} ( z - c_i(t))$. Pick any sufficiently divisible integer $N$ such that 
all series $c_i(t^N)$ become formal power series. By Artin's approximation theorem~\cite{Artin}, they are necessarily analytic
in a neighborhood of $0$.

Our Main Theorem applied to the meromorphic family $P_{t^N}$ and  the marked points $c_i(t^N)$ shows that we can write
$g_{P_{t^N}} (c_i(t^N)) = \lambda_i \log |t|^{-1} + h_i(t)$ where $h_i$ is a continuous function and $\lambda_i \in \Q_+$.
By~\eqref{eq:crit mis}, we infer
$$
L(P_{t^N}) = \log |d| +\left(\sum_{i=1}^{d-1} \lambda_i\right)\,  \log|t|^{-1}  + \sum_{i=1}^{d-1} h_i(t)~.
$$
Now observe that by definition $\tilde{h}(t) =  \sum_{i=1}^{d-1} h_i(t)$ is a continuous function on the unit disk
which is invariant by the multiplication by any $N$-th root of unity. It follows that one may find a continuous function 
$h$ on the unit disk such that $h(t^N) = \tilde{h}(t)$, and we get $L(P_t) = \lambda \log|t|^{-1} + h(t)$ with 
$\lambda := \frac1N\, \sum_{i=1}^{d-1} \lambda_i\in \Q_+$ as required.
 
\smallskip

Suppose now that $\lambda =0$. Denote by $\sP$ the polynomial with coefficients in $k((t))$ defined by the family $P_t$, and  by
$\sc_i$ the point in $k((t^{1/N}))$ defined by the Puiseux series $c_i(t)$.
By~\cite[Theorem~C]{favre} and~\cite[\S 5]{okuyama}, it follows that $g_{\sP}(\sc_i) =0$ for all $i$ so that all critical points of $\sP$ belongs to the filled-in Julia set. 

We claim that there exists  an affine transformation $\phi$ with coefficients in $k((t))$  such that 
$\sQ := \phi^{-1}\circ \sP\circ \phi$ leaves the closed unit ball totally invariant. 

Granting this claim we conclude the proof of the corollary. We write 
$\phi(z) = b_0(t)z + b_1(t)$  with $b_i(t) = t^{-n_i} (b_{i0} + \sum_{i\ge 1} b_{ij} t^j)$, $b_{i0} \neq0$ and $b_{ij}\in k$, 
and we define for all $M\ge1$ the affine transformation obtained by truncation with coefficients in the field of Laurent polynomials $\phi_M = b^M_0(t)z + b^M_1(t)$  such that
$b_i^M =   t^{-n_i} (b_{i0} + \sum_{M\ge i\ge 1} b_{ij} t^j)$.
For $M$ large enough the difference $\sQ - \sQ_M$ is a polynomial with coefficients in $t k[[t]]$ so that 
the polynomial  $\sQ_M := \phi_M^{-1}\circ \sP\circ \phi_M$ leaves the closed unit ball totally invariant too.

In particular  $\sQ_M$ is a polynomial of degree $d$ with coefficients in $k[[t]]$
and dominant term $\sa z^d$ with $\sa$ invertible (in $k[[t]]$). Together with the fact that the family $P_t$ is meromorphic and the 
coefficients of $\phi_M$ are Laurent polynomials,  we conclude that $\sa$ determines an analytic function $a(t)$ with $a(0) \neq 0$, and
$\sQ_M$ determines an analytic family of polynomials of the form  $Q_t (z) = a(t) \, z^ d  + \mathrm{l.o.t}$. conjugated to $P_t$, as required.

\smallskip

It remains to prove our claim. Denote by $\mathbb{L}$ the completion of the field of Puiseux series (i.e. of the algebraic closure of $k((t))$).

By~\cite[Corollary 2.11]{Kiwi:Cubic} the fact that all critical points of $\sP$ belong to the filled-in Julia set implies that $\sP$ is simple over $\mathbb{L}$.
This means that the filled-in Julia set of $\sP$ in $\A^1_\mathbb{L}$ is equal to a closed ball $\bar{B}$. 
This ball contains all fixed points, and this set is defined by a polynomial equation of the form $a_0 z^l + a_1 z^{l-1} + \ldots + a_l  =0$ of degree $l$ with $a_i\in k((t))$ hence
$B$ contains the point $ - a_1/ (l a_0) \in k((t))$. The radius of $B$ also belongs to $|k((t))^*|$ because $B$ is fixed by 
$\sP$, so that we can find an affine transformation with coefficients in $k((t))$ sending $B$ to the closed unit ball.

\section{The adelic metric associated to a pair $(P,a)$}

In this section, we prove Corollary~\ref{cor2}. Let us first recall the setting.
Let $C$ be any smooth connected affine curve $C$ defined over a number field $\KK$.
Assume $P$ is an algebraic family parametrized by $C$ and $a\in\KK[C]$ is a marked point such that $(P,a)$ is not isotrivial, and $a$ is not persistently preperiodic. 

Denote by $M_\KK$ be the set of places of  $\KK$.

\medskip

\noindent{ \bf Step 1}: construction of  a  suitable line bundle $\mathcal{L}$ on $\bar{C}$ the (smooth) projective compactification of $C$.

\smallskip

To any branch at infinity $\mathfrak{c}\in \bar{C} \setminus C$ we associate a non-negative rational number $\alpha(\mathfrak{c})$ as follows. 

We fix a projective embedding of $\bar{C}$ into the projective space $\p^3_\KK$ such that $\mathfrak{c}$ is the homogeneous point $[0:0:0:1]$. By e.g.~\cite[Proposition 3.1]{favre-gauthier}, there exist  a number field $\LL \supset \KK$, a finite set of places $S$ of $\LL$ and adelic series $\beta_1, \beta_2, \beta_3 \in t \cO_{\LL,S}[[t]]$ such that the following holds.

For each place $v \in M_{\LL}$ the series $\beta_j(t)$
are convergent in some neighborhood 
$\{ |t| < c_v\}$ of the origin.  The map $\beta(t) = [\beta_1(t): \beta_2(t): \beta_3(t)):1]$ induces an analytic isomorphism
from $\{ |t| < c_v\}$ to a neighborhood of $\mathfrak{c}$ in $\bar{C}(\C_v)$.  And the constants $c_v $ equal $1$ for all but finitely many places. In the sequel we refer to an adelic parameterization of $\bar{C}$ near $\mathfrak{c}$ for such a data.

\smallskip

Our family of polynomials is determined by $d+1$ rational functions on $\bar{C}$  
$$P(z) = \a_0 z^d + \a_1 z^{d-1} + \cdots + \a_d$$
with $\a_i \in \KK(C)$, so that  $P_\mathfrak{c} :=  (\a_0\circ \beta)\, z^d + (\a_1\circ \beta)\, z^{d-1} + \cdots + (\a_d\circ \beta)$
belongs to $\cO_{\LL,S}((t))[z]\subset \LL((t))[z]$. Write $a_\mathfrak{c} = a \circ \beta \in \LL((t))$.

Working over the non-Archimedean field $L = \LL((t))$ endowed with the $t$-adic norm, we
may define $\a(\mathfrak{c}) := g_{P_\mathfrak{c}}(a_\mathfrak{c})$ which is a non-negative rational number by Theorem~\ref{thm:trucco}.

\smallskip

We finally define the effective divisor  with rational coefficients
$$\mathsf{D} := \sum_{\bar{C}\setminus C} \a(\mathfrak{c}) [\mathfrak{c}]$$
 and set
$\mathcal{L} := \cO_{\bar{C}} (q \mathsf{D})$ for a sufficiently divisible $q\in \N^*$.
Observe that since $C$ is defined over $\KK$, its projective compactification $\bar{C}$ is also defined over $\KK$ and the divisor $\mathsf{D}$ too since it is invariant by the absolute Galois group of $\KK$.

\medskip

\noindent{ \bf Step 2}: we build a semi-positive and continuous metrization $|\cdot|_{\mathcal{L},v}$ on the line bundle induced by $\mathcal{L}$ on $\bar{C}_{\KK_v}$ for any place $v\in M_\KK$.

\smallskip

Fix a place $v\in M_\KK$. We let $\C_v$ be the completion of the algebraic closure $\bar{\KK}_v$ of the completion $\KK_v$ of $(\KK,|\cdot|_v)$, and define
\[g_{P_t,v}(z):=\lim_{n\to\infty}\frac{1}{d^n}\log^+|P_t^n(z)|_v, \ z\in\C_v.\]
Pick a branch at infinity $\mathfrak{c}\in \bar{C}\setminus C$ and choose a local adelic parameterization $\beta$ of $\bar{C}$ centered at $\mathfrak{c}$ as in the previous step. It is given by formal power series with coefficients in a number field $\LL$.

According to the Main Theorem, there exists $\alpha_v(\mathfrak{c})\in\Q_+$ such that 
\begin{equation}
g_{a,v}(t):=g_{P_{\beta(t)},v}(a(t))=\alpha_v(\mathfrak{c})\cdot\log|t|^{-1}+h_{\mathfrak{c},v}(t),\label{eq:cont}
\end{equation}
where $h_{\mathfrak{c},v}$ extends continuously across $0$. 
Moreover by~\eqref{eq:good exp} the constant  $\alpha_v(\mathfrak{c})$ is equal to $g_{P_\mathfrak{c}}
(a_\mathfrak{c}) = \alpha(\mathfrak{c})$. In particular,  $\alpha_v(\mathfrak{c})$ is independent  of $v$. 

Pick an open subset $U$ of the Berkovich analytification $\bar{C}^{v,\an}$ of $\bar{C}$ over the field $\KK_v$ and a section $\sigma$ of the line bundle $\mathcal{L}$ over $U$. By definition, $\sigma$ is a meromorphic function on $U$ whose divisor of poles and zeroes satisfies $\mathrm{div}(\sigma)+q \mathsf{D}\geq0$. We set
\[|\sigma|_{a,v}:=|\sigma|_ve^{-q\cdot g_{a,v}}.\]
According to \eqref{eq:cont}, the function $|\sigma|_{a,v}$ is continuous. Moreover, since $g_{a,v}$ is subharmonic on $C^{v,\an}$ and the function $-\log|\sigma|_{a,v}$ is subharmonic on $U\cap C^{v,\an}$. Since it extends continuously to $U$, \cite[ Lemma 3.7]{favre-gauthier} implies that $-\log|\sigma|_{a,v}$ is subharmonic on $U$. 
This implies the metrization $|\cdot|_{a,v}$ is continuous and semi-positive in the sense of Zhang (by definition in the Archimedean case and by \cite[Lemma 3.11]{favre-gauthier} in the non-Archimedean case).

%

\medskip

\noindent{ \bf Step 3}: the line bundle  $\mathcal{L}$ is  (very) ample (if $q$ is large enough). 

\smallskip

Since $\mathcal{L}$ is determined by the effective divisor $\mathsf{D}= \sum_{\bar{C}\setminus C}\alpha(\mathfrak{c})\cdot [\mathfrak{c}]$
it is sufficient to show that $\alpha(\mathfrak{c})>0$ for at least one branch at infinity.
Suppose by contradiction that $\mathsf{D}=0$, and choose any Archimedean place $v_0\in M_\KK$. 
Observe first that the function $g_{a,v_0}$ extends continuously to $\bar{C}(\C)$ as a subharmonic function which is thus constant since $\bar{C}(\C)$ is compact.

By~\cite{favredujardin} the family of analytic maps $t \mapsto P_t^n(a(t))$ is hence normal locally near any point
$t\in C$. Since $(P,a)$ not isotrivial, \cite[Theorem 1.1]{demarco} implies $a$ is persistently preperiodic, which is a contradiction.

\medskip

\noindent{ \bf Step 4}: the collection of metrizations $|\cdot|_{\mathcal{L},v}$ equips $\mathcal{L}$ with an adelic semi-positive continuous metrization, and 
\begin{equation}\label{eq:easy}
h_{\hat{\mathcal{L}}}(t)=q\cdot h_{P,a}(t) \ \text{ for all } \ t\in C(\bar{\KK}).
\end{equation}

Let us prove the first assertion. Since for any place $v$ the metrization $|\cdot|_{\mathcal{L},v}$ is semi-positive and continuous, 
one only needs to  show that the collection $\{|\cdot|_{a,v}\}_{v\in M_\KK}$ is adelic. Following exactly the proof of \cite[Lemma 4.2]{ghioca-ye}, we get the existence of $g\in\KK(\bar{C})$ such that $q\cdot g_{a,v}(z)=\log|g(z)|_v$ for all but finitely many places $v\in M_\KK$ and the conclusion follows.

To get~\eqref{eq:easy}, we follow closely \cite[\S4.1]{favre-gauthier}. If $t$ is a point in $C$ that is defined over a finite extension $\KK$, 
denote by $\mathsf{O}(t)$ its orbit  under the absolute Galois group of $\KK$, and let $\deg(t):=\mathrm{Card}(\mathsf{O}(t))$.
Fix a rational function $\phi$ on $\bar{C}$ with $\mathrm{div}(\phi)+q\mathsf{D}\geq0$ that is not vanishing at $t$.
By~\cite[\S 3.1.3]{ACL2}, since $\phi(t)\neq0$ we have
\begin{eqnarray*}
h_{\hat{\mathcal{L}}}(t)
&= &
\frac1{\deg(t)} \sum_{t'\in\mathsf{O}(t)} \sum_{v\in M_\KK}  - \log |\phi|_{a,v}(t')
\\
&=&
\frac1{\deg(t)} \sum_{t'\in\mathsf{O}(t)} \sum_{v\in M_\KK}  (q\cdot g_{a,v}(t') -  \log |\phi|_v(t'))
\\
&=&
\frac1{\deg(t)} \sum_{t'\in \mathsf{O}(t)} \sum_{v\in M_\KK}  q\cdot g_{P_{t'},v} (a(t'))= q\cdot \hat{h}_{P_t}(a(t))\ge 0
\end{eqnarray*}
where the last line follows from the product formula and the definition of $\hat{h}_{P_t}$. 

\medskip

\noindent{ \bf Step 5}: the total height of $\bar{C}$ is $h_{\hat{\mathcal{L}}}(\bar{C})=0$.

\smallskip

We use~\cite[(1.2.6) \& (1.3.10)]{ACL2}.
Choose any two meromorphic functions $\phi_0, \phi_1$  such that 
$\mathrm{div}(\phi_0)+ q\mathsf{D}$ and $\mathrm{div}(\phi_1)+ q\mathsf{D}$ are both effective with disjoint support included in $C$.  Let $\sigma_0$ and $\sigma_1$ be the associated sections of $\mathcal{O}_{\bar{C}}(q\mathsf{D})$. Let $\sum n_i [t_i]$ be the divisor of zeroes of $\sigma_0$, and
$\sum n'_j [t'_j]$ be the divisor of zeroes of $\sigma_1$.
Then 
\begin{eqnarray*}
h_{\hat{\mathcal{L}}}(\bar{C}) 
& = &
\sum_{v\in M_\KK} (\hdiv(\sigma_0)\cdot \hdiv(\sigma_1) | \bar{C})_v
\\
& = &
\sum_i n_i \cdot q\cdot \hat{h}_{P_{t_i}}(a(t_i)) - 
\sum_{v\in M_\KK} \int_{\bar{C}} \log| \sigma_0|_{a,v}\,  \Delta (q\cdot g_{a,v})
\\
& = &
\sum_{v\in M_\KK} \int_{\bar{C}} q\cdot g_{a,v}\,  \Delta (q\cdot g_{a,v}) \ge0~,
\end{eqnarray*}
where the third equality follows from Poincar\'e-Lelong formula and writing 
$\log| \sigma_0|_{a,v} = \log |\phi_0|_v - q\cdot g_{a,v}$.

\smallskip

Pick any archimedean place $v_0$. The positive measure $\Delta g_{a,v_0}$ has total mass the degree of $\mathcal{L}$ on $\bar{C}$ hence is non-zero. It follows from e.g.~\cite[Lemma 2.3]{favredujardin} that any point $t_0$ in the support of $\mu_{a,v_0}$ is accumulated by parameters $t_* \in C(\KK)$ such that $P_{t_*}^n(a(t_*))=P_{t_*}^m(a(t_*))$
 for some $n>m\geq0$. For any such point~\eqref{eq:easy} implies
$h_{\hat{\mathcal{L}}}(t_*)=0$.  In particular, the essential minimum of $h_{\hat{\mathcal{L}}}$  is non-positive. 
By the arithmetic Hilbert-Samuel theorem (see~\cite[Th\'eor\`eme~4.3.6]{thuillier},~\cite[Proposition~3.3.3]{MR1810122}, or ~\cite[Theorem 5.2]{zhang}), we get $h_{\hat{\mathcal{L}}}(\bar{C}) = 0$, ending the proof.

\bibliographystyle{short}
\bibliography{biblio}

\begin{thebibliography}{FRL}

\bibitem[Ar]{Artin}
M.~Artin.
\newblock On the solutions of analytic equations.
\newblock {\em Invent. Math.}, 5:277--291, 1968.

\bibitem[Au]{MR1810122}
Pascal Autissier.
\newblock Points entiers sur les surfaces arithm\'etiques.
\newblock {\em J. Reine Angew. Math.}, 531:201--235, 2001.

\bibitem[B]{Berko}
Vladimir~G. Berkovich.
\newblock {\em Spectral theory and analytic geometry over non-{A}rchimedean
  fields}, volume~33 of {\em Mathematical Surveys and Monographs}.
\newblock American Mathematical Society, Providence, RI, 1990.

\bibitem[BD1]{BD-unlikely}
Matthew Baker and Laura DeMarco.
\newblock Preperiodic points and unlikely intersections.
\newblock {\em Duke Math. J.}, 159(1):1--29, 2011.

\bibitem[BD2]{BD}
Matthew Baker and Laura DeMarco.
\newblock Special curves and postcritically finite polynomials.
\newblock {\em Forum Math. Pi}, 1:e3, 35, 2013.

\bibitem[CL]{ACL2}
Antoine Chambert-Loir.
\newblock Heights and measures on analytic spaces. {A} survey of recent
  results, and some remarks.
\newblock In {\em Motivic integration and its interactions with model theory
  and non-{A}rchimedean geometry. {V}olume {II}}, volume 384 of {\em London
  Math. Soc. Lecture Note Ser.}, pages 1--50. Cambridge Univ. Press, Cambridge,
  2011.

\bibitem[D]{demarco}
Laura DeMarco.
\newblock Bifurcations, intersections, and heights.
\newblock {\em Algebra Number Theory}, 10(5):1031--1056, 2016.

\bibitem[DF]{favredujardin}
Romain Dujardin and Charles Favre.
\newblock Distribution of rational maps with a preperiodic critical point.
\newblock {\em Amer. J. Math.}, 130(4):979--1032, 2008.

\bibitem[DG]{DeMarco-Ghioca}
L.~{DeMarco} and D.~{Ghioca}.
\newblock {Rationality of dynamical canonical height}.
\newblock {\em ArXiv e-prints}, February 2016.

\bibitem[F]{favre}
C.~{Favre}.
\newblock {Degeneration of endomorphisms of the complex projective space in the
  hybrid space}.
\newblock {\em ArXiv e-prints}, November 2016.

\bibitem[FG]{favre-gauthier}
C.~{Favre} and T.~{Gauthier}.
\newblock {Classification of Special Curves in the Space of Cubic Polynomials}.
\newblock {\em International Math Research Notices, to appear}, 2016.

\bibitem[FRL]{MR2578470}
Charles Favre and Juan Rivera-Letelier.
\newblock Th\'eorie ergodique des fractions rationnelles sur un corps
  ultram\'etrique.
\newblock {\em Proc. Lond. Math. Soc. (3)}, 100(1):116--154, 2010.

\bibitem[GY]{ghioca-ye}
D.~{Ghioca} and H.~{Ye}.
\newblock {A Dynamical Variant of the {A}ndr\'e-{O}ort {C}onjecture}.
\newblock {\em International Math Research Notices, to appear}, 2016.

\bibitem[K]{Kiwi:Cubic}
Jan Kiwi.
\newblock Puiseux series polynomial dynamics and iteration of complex cubic
  polynomials.
\newblock {\em Ann. Inst. Fourier (Grenoble)}, 56(5):1337--1404, 2006.

\bibitem[O]{okuyama}
Y\^usuke Okuyama.
\newblock Quantitative approximations of the {L}yapunov exponent of a rational
  function over valued fields.
\newblock {\em Math. Z.}, 280(3-4):691--706, 2015.

\bibitem[S]{Silverman}
Joseph~H. Silverman.
\newblock {\em The arithmetic of dynamical systems}, volume 241 of {\em
  Graduate Texts in Mathematics}.
\newblock Springer, New York, 2007.

\bibitem[Th]{thuillier}
Amaury Thuillier.
\newblock Th\'eorie du potentiel sur les courbes en g\'eom\'etrie analytique
  non archim\'edienne. applications \`a la th\'eorie d'arakelov, 2005.
\newblock Th\`ese de l'Universit\'e de Rennes 1, viii + 184 p.

\bibitem[Tr]{trucco}
Eugenio Trucco.
\newblock Wandering {F}atou components and algebraic {J}ulia sets.
\newblock {\em Bull. Soc. Math. France}, 142(3):411--464, 2014.

\bibitem[Z]{zhang}
Shou-Wu Zhang.
\newblock Positive line bundles on arithmetic varieties.
\newblock {\em J. Amer. Math. Soc.}, 8:187--221, 1995.

\end{thebibliography}

\end{document}